\documentclass[a4paper,11pt]{article}

\author{J. C. Meyer}
\usepackage{fullpage}
\usepackage[top=30mm, bottom=30mm, left=25mm, right=25mm]{geometry}
 
\title{A note on boundary point principles for partial differential inequalities of elliptic type}

\usepackage{amsmath, amsthm, amssymb, amsfonts, MnSymbol}
\usepackage{hyperref}

\allowdisplaybreaks

\newtheorem{thm}{Theorem}[section]

\newtheorem{lem}[thm]{Lemma}

\theoremstyle{definition}
\newtheorem{defn}[thm]{Definition}
\theoremstyle{definition}
\newtheorem{remk}[thm]{Remark}
\theoremstyle{definition}
\newtheorem{ex}[thm]{Example}
\theoremstyle{definition}
\numberwithin{equation}{section}

\begin{document}
\maketitle

\begin{abstract}
In this note we consider boundary point principles for partial differential inequalities of elliptic type.
Firstly, we highlight the difference between conditions required to establish classical strong maximum principles and classical boundary point lemmas for second order linear elliptic partial differential inequalities.
We highlight this difference by introducing a singular set in the domain where the coefficients of the partial differential inequality need not be defined, and in a neighborhood of which, can blow-up.  
Secondly, as a consequence, we establish a comparison-type boundary point lemma for classical elliptic solutions to quasi-linear partial differential inequalities.
Thirdly, we consider tangency principles, for $C^1$ elliptic weak solutions to quasi-linear divergence structure partial differential inequalities.
We highlight the necessity of certain hypotheses in the aforementioned results via simple examples.
\end{abstract}

\section{Introduction}
In this note we consider boundary point principles (BPP) for solutions to elliptic partial differential inequalities (PDI).
Specifically, we first give a relaxation of Hopf's \cite{EH1} classical strong maximum principle (CSMP) for classical solutions to linear elliptic PDI, which here, allows the coefficients in the PDI to be unbounded in a neighborhood of a sufficiently regular subset of the spatial domain. 
The boundary point lemma (BPL) for linear elliptic PDI is obtained a consequence of this CSMP and complements available results (see \cite{GL1} and \cite{PPJS1}).  
Although coefficients in the PDI in the BPL are not necessarily bounded, they are constrained by growth conditions detailed in Section 2. 
As a secondary consideration, we illustrate how to extend BPL for classical solutions to linear elliptic PDI, to comparison-type BPL for elliptic classical solutions of quasi-linear PDI, and highlight the importance of specific conditions on this extension via examples.
Consequently, we demonstrate that the BPL, as stated in \cite[Theorem 2.7.1]{PPJS1} is erroneous.
Thirdly, we give an extension of a tangency principle for $C^1$ elliptic weak solutions to divergence structure quasi-linear PDI in domains with boundaries that satisfy an interior cone condition, which appeared in \cite{JS1}.
We also highlight that the tangency principle in \cite[Theorem 2.7.2]{PPJS1}, an extension of that in \cite{JS1}, is erroneous.
Corrections to the aforementioned erroneous theorem statements are provided.

We now give a brief account of the historical development of results in this note. 
The CSMP and BPL for classical solutions to linear elliptic PDI were established by Hopf in \cite{EH1} for linear elliptic PDI with bounded (uniformly elliptic) coefficients.
Although Hopf considered generalizations of the CSMP and BPL to elliptic solutions of nonlinear PDI in \cite{EH1}, more general statements of these results were established by McNabb in \cite{AM1}.
Extensions to the CSMP and BPL for classical solutions to linear elliptic PDI with coefficients that can blow-up or degenerate have been considered by numerous authors, as summarised in \cite{AN1}, \cite{GL1} and \cite{PPJS2}.
Additionally, due to the the development of a theory for weak solutions to boundary value problems for divergence structure quasi-linear elliptic PDE, tangency principles for $C^1$ elliptic weak solutions to quasi-linear PDI were established by Serrin in \cite{JS1}, and extended in \cite{PPJS1}.   
We note that the proof of Serrin relies on an iteration method developed by Moser \cite{JM1} and a Harnack inequality for quasi-linear divergence structure elliptic PDI established by Trudinger \cite{NT1}.
More recently tangency principles have been established for $C^1$ elliptic weak solutions to quasi-linear PDI which have conclusions more similar to that of Hopf type BPL (see \cite{LR2}, \cite{LR1} and \cite{SDL1}).
A broader historical overview of the development of this theory can be found in \cite[p.156-158 and p.193-194]{PW1}, \cite[p.46]{PPJS1}, \cite{AIN1}, \cite{ABMMZ1} and \cite{PPVR1}. 

The remainder of the note is presented as follows.
In Section 2, we prove the CSMP and BPL for classical solutions to linear elliptic PDI, and consequently, we establish a comparison-type BPL for elliptic classical solutions to quasi-linear PDI.
Furthermore, we provide examples which highlight the need for specific conditions given in the statement of the BPL as given here, one of which, is a counter-examples to \cite[Theorem 2.7.1]{PPJS1}.
In Section 3, we establish a comparison-type tangency principle for $C^1$ elliptic weak solutions to quasi-linear divergence structure PDI in domains which satisfy an interior cone condition at boundary points.
The necessity of several conditions in the BPL statement are highlighted, and furthermore, we demonstrate that \cite[Theorems 2.7.2 and 2.7.3]{PPJS1}, are erroneous.
In Section 4, we discuss how results in this note can be generalised and placed in a wider context.

\section{Classical Theory}

In this section, we establish a CSMP in Theorem \ref{CSMP} and BPL in Theorem \ref{BPL1} for classical solutions to linear elliptic PDI.
The CSMP is noteworthy in that it allows coefficients in the PDI, under constraint, to blow-up in the interior of the domain in the neighborhood of a singular set.
After defining the regularity of the singular set and constructing a suitable auxiliary function, the proofs of these results largely follow the description of related proofs available in \cite[Chapter 2]{PPJS1}.
This allows us to highlight a distinction between the conditions required to establish a CSMP and BPL for classical solutions to linear elliptic PDI.
Consequently, we also establish a comparison type BPL for classical elliptic solutions to quasi-linear PDI in Theorem \ref{t2} using the aforementioned BPL for linear elliptic PDI, refining an analogous statement in \cite[Theorem 2.7.1]{PPJS1}.
We provide a proof using the approach outlined in \cite[Section 2.7]{PPJS1} where it is noteworthy that a full proof is omitted. 
To conclude the section, we give a simple counter-example to \cite[Theorem 2.7.1]{PPJS1} and provide a further example to highlight the importance of specific conditions in Theorem \ref{t2} which are not present in \cite[Theorem 2.7.1]{PPJS1}.

\subsection{Notation and Definitions}
For a set $X\subset \mathbb{R}^n$, we denote $\partial X=\bar{X} \setminus \text{int}(X)$, to be the boundary of $X$. 
In addition, throughout this note, $\Omega\subset\mathbb{R}^n$ denotes an open connected bounded set (a bounded domain), and we denote the set $B_R(x_0)\subset\mathbb{R}^n$ to be an open n-dimensional ball of radius $R$ (with respect to the Euclidean distance) centred at $x_0\in\mathbb{R}^n$.
Furthermore, we denote $R(X)$ to be the set of real-valued functions with domain $X$, $C(X)\subset R(X)$ to be the set of all continuous functions in $R(X)$ and $C^i(X)\subset C(X)$ to be the set of $i$-times continuously differentiable functions in $C(X)$ for each $i\in\mathbb{N}$.  
Additionally, for $u\in C^2(\Omega )$ and $\mathcal{S}\subset \Omega$, we consider the linear elliptic operator $L:C^2(\Omega )\to R(\Omega \setminus \mathcal{S} )$ given by
\begin{equation} \label{nc1} 
L[u]:= \sum_{i,j=1}^n a_{ij}u_{x_ix_j} + \sum_{i=1}^n b_iu_{x_i} + cu \ \ \ \text{in } \Omega \setminus \mathcal{S} , 
\end{equation} 
with $a_{ij},b_i,c:\Omega \setminus \mathcal{S} \to\mathbb{R}$ prescribed functions for $i,j=1,\dots ,n$, and such that there exists a non-negative function $\Lambda :\Omega \setminus \mathcal{S}\to\mathbb{R}$ for which,
\begin{equation} \label{nc2} 
|y|^2 \leq \sum_{i,j=1}^n a_{ij}(x)y_iy_j \leq \Lambda (x)|y|^2 \ \ \ \forall x\in\Omega \setminus \mathcal{S} , \ y\in\mathbb{R}^n . 
\end{equation}
We refer to the set $\mathcal{S}$ where the linear elliptic operator is not defined for $u$, as the {\emph{singular set}}.
Additionally, note that by re-scaling the coefficients in the operator in \eqref{nc1} by $\epsilon$, the left hand side of \eqref{nc2} can be expressed as $\epsilon |y|^2$ i.e. with an equivalent frequently used ellipticity condition.
Moreover, for $u\in C^2(\Omega )$ we denote $Du$ and $D^2u$ to be the gradient of $u$ and the Hessian of $u$ on $\Omega$, respectively.

To establish the CSMP in this note, we give the following definition, which will be used to define the structure of the singular set $\mathcal{S}\subset \Omega$.
We refer to $\mathcal{S}$ as the singular set since the coefficients $a_{ij}$, $b_i$ or $c$ of $L$ are allowed, with constraint, to blow up in neighborhoods of $\mathcal{S}$.   
We note that in \cite{ABMMZ1}, alternatively, two-sided `hour glass' conditions are employed for regularity conditions on singular sets which complement the following definition.

\begin{defn} \label{D1}
Let $\Omega\subset\mathbb{R}^n$ be a domain and $S\subset \Omega$. We say that $\mathcal{S}$ satisfies an {\emph{outward ball property}} if, given any nonempty relatively closed set $\mathcal{T}\subset \Omega$ that is a strict subset of $\Omega$, there exists $R>0$ and $x_0\in \Omega \setminus (\mathcal{T}\cup \mathcal{S})$ such that
\begin{equation} \label{D1a} 
B_R(x_0)\subset \Omega \setminus (\mathcal{T}\cup \mathcal{S})\ \text{ and }\ \partial B_R(x_0) \cap \mathcal{T} \not= \emptyset . 
\end{equation}
\end{defn}

To illustrate some geometric aspects of sets that satisfy an outward ball property, consider the following:

\begin{itemize}
\item[(i)] If $\mathcal{S}$ consists solely of a finite number of points in $\Omega$ then $\mathcal{S}$ satisfies the outward ball property.
This follows by considering $d_{H'}:\mathcal{P}(\mathbb{R}^n)\times \mathcal{P}(\mathbb{R}^n)\to [0,\infty )$ with $\mathcal{P}(X)$ denoting the power set of $X$, and
\[ 
d_{H'}(X,Y) = \sup_{x\in X} \left( \inf_{y\in Y} |x-y| \right) \ \ \ \forall X,Y\in \mathcal{P}(\mathbb{R}^n), 
\]
i.e. one component of the Euclidean Hausdorff distance between $X$ and $Y$. 
Note that if $|X|=1$, then $d_{H'}$ is the Euclidean Hausdorff distance between the two sets $X$ and $Y$, denoted here by $d(X,Y)$. 
Now, let $\mathcal{T}$ be as in Definition \ref{D1}. 
Then since $\mathcal{T}$ is nonempty and $\mathcal{T}\not=\Omega$, it follows that $\partial \mathcal{T}\cap \Omega \not=\emptyset$. 
If $d_{H'}(\partial \mathcal{T} \cap \Omega , \mathcal{S})=0$, it follows that $\mathcal{T} \subseteq  \mathcal{S}$, and we can choose a point $x_0\in\Omega \setminus ( \mathcal{T}\cup \mathcal{S})$ sufficiently close to $\mathcal{T}$ such that there exists a ball $B_R(x_0)$ that satisfies \eqref{D1a}. 
Alternatively, if $d_{H'}(\partial \mathcal{T} \cap \Omega ,  \mathcal{S})>0$, then since $\Omega \setminus (\mathcal{T}\cup  \mathcal{S})$ is a nonempty open set, we can chose $x_0\in \Omega \setminus (\mathcal{T}\cup  \mathcal{S})$ so that $d_{H'}(\{ x_0\} ,\mathcal{T})<\tfrac{1}{2} d_{H'}(\{ x_0\} , \mathcal{S}\cup \partial \Omega )$.
Thus, there exists a ball $B_R(x_0)$ that satisfies \eqref{D1a}. 
\item[(ii)] If $\Omega =(-1,1)^2\subset \mathbb{R}^2$ and 
\begin{align*}
\mathcal{S}=&\{ (x_1,x_2)\in \Omega :\ (x_1,x_2) = (\phi_1(t),\phi_2(t)) \ \ \forall t\in (0,1)\text{ with } \phi : (0,1)\to\Omega \text{ twice } \\ 
&\  \text{continuously differentiable and injective on }(0,1) \text{ with }\lim_{t\to 0}\phi (t)=\phi_0 \not= \phi_1 = \lim_{t\to 1} \phi(t) \} ,
\end{align*}
then $\mathcal{S}$ satisfies the outward ball property. 
To see this, let $\mathcal{T}$ be as in Definition \ref{D1}. 
If $d_{H'}(\partial T \cap \Omega , \mathcal{S})>0$, then a ball that satisfies \eqref{D1a} is guaranteed to exist, following the justification in (i). 
Alternatively, if $d_{H'}(\partial T \cap \Omega, \mathcal{S})=0$, then it follows that $\partial T \cap \Omega \subseteq \overline{\mathcal{S}}$. 
Suppose $\partial \mathcal{T}\cap \mathcal{S}\supset \{ s_0\}$.
Then since $\mathcal{S}$ is given by a sufficiently smooth curve, for $s_0$, there exists a ball $B_R(s_0)\subset \Omega$ such that
\[ 
\mathcal{S} \cap  \bar{B}_R(s_0) = \{ (\phi_1(t),\phi_2(t)):t_1 \leq t \leq t_2 \} =: \mathcal{S}_R \cup \{ (\phi_1(t_1),\phi_2(t_1)),\ (\phi_1(t_2),\phi_2(t_2)) \} .
\]
Thus, $\partial \mathcal{T}\cap B_R (s_0) \subset \mathcal{S}_R$ and hence, via the Jordan Curve Theorem, $B_R(s_0)$ can be decomposed into the disjoint sets $\mathcal{S}_R$, $B_R^{1}(s_0)$ and $B_R^{2}(s_0)$ with $B_R^{1}(s_0)$ the connected open set with boundary $\mathcal{S}_R$ and the arc on $\partial B_R(s_0)$ connecting $(\phi_1(t_1),\phi_2(t_1))$ to $(\phi_1(t_2),\phi_2(t_2))$ in a clockwise direction ($B_R^{2}(s_0)$ is defined similarly to $B_R^{1}(s_0)$ with clockwise replaced by anti-clockwise). 
Thus, $\mathcal{T} \cap B_R(s_0)$ is either: $\mathcal{S}_R\cap \mathcal{T}$, $\mathcal{S}_R\cup B_R^{1}(s_0)$ or $\mathcal{S}_R \cup B_R^{2}(s_0)$, and in each case, since $\mathcal{S}_R$ is defined by a $C^2$ curve, there exists a ball $B_{R_1}(x)\subset B_R(s_0) \setminus (\mathcal{T}\cup \mathcal{S})$ that satisfies \eqref{D1a}.
If instead  $d_{H'}(\partial \mathcal{T} \cap \Omega, \mathcal{S})=0$ and $\partial \mathcal{T}\cap \mathcal{S} = \emptyset$ then a similar argument to that in (i) can be used to demonstrate that a ball that satisfies \eqref{D1a} exists.
It follows analogously from the Jordan-Brouwer Separation Theorem that any set of finitely many disjoint compact $(n-1)-$dimensional sufficiently smooth $C^2$ manifolds in a domain $\Omega\subset\mathbb{R}^n$ also satisfies the outward ball property. 
\item[(iii)] If $\Omega = (-1,1)^2$ and 
\[ 
\mathcal{S}' = \{ (x_1,x_2)\in \Omega :\ x_1=0 \text{ or }x_2=0\}, 
\]
then $\mathcal{S}'$ does not satisfy the outward ball property. 
This follows by considering  $\mathcal{T}=\{ (0,0) \}$ and observing that every ball $B_R(x)\subset \Omega$ such that $\partial B_R(x)\cap \mathcal{T} \not = \emptyset$, also satisfies $B_R(x)\cap \mathcal{S}' \not=\emptyset$. 
However, if  instead $\Omega = (-1,1)^2$ and 
\[ 
\mathcal{S} = \{ (x_1,x_2)\in [0,1)\times (-1,1) :\ x_1=0 \text{ or }x_2=0\} 
\]
then $\mathcal{S}$ satisfies the outward ball property.
\item[(iv)] If $\mathcal{S}'$ is locally dense on $B_R(x)\subset \Omega$ then $\mathcal{S}'$ does not satisfy the outward ball property.
This can be observed by choosing $\mathcal{T}$ to contain any point in $\mathcal{S}'\cap B_R(x)$. 
Consequently, sets that satisfy Definition \ref{D1} are necessarily measure zero sets with respect to the Lebesgue measure.
\item[(v)] If $\mathcal{S}$ satisfies Definition \ref{D1}, then $\mathcal{S}$ is $1$-porous at each $s\in\mathcal{S}$ with respect to \cite[Definition 2.1]{Z1}.
This follows by considering $\mathcal{T}=\{ s\}$. 
However, not all subsets of $\Omega$ that are $1$-porous at every point necessarily satisfy Definition \ref{D1}.
For example, consider $\Omega = (-1,1)^2$ with 
\[ 
\mathcal{S}' = \left\{ (x_1,x_2)\in \Omega :\ x_1 =\tfrac{1}{2n} \text{ for } n\in\mathbb{N} \right\} .
\]
Since $\mathcal{S}'$ consists of a countable set of isolated lines, it follows immediately that $\mathcal{S}'$ is $1$-porous at each $s\in \mathcal{S}'$.
However by considering $\mathcal{T}= (-1,0]\times (-1,1)\subset \Omega$, it follows that $\mathcal{S}'$ does not satisfy the outward ball property.
It is noteworthy that the review articles \cite{Z2} and \cite{Z1} do not indicate that a link has been established between porous sets and singular sets for elliptic PDI.
\end{itemize}

Later in this section, for $u\in C^2(\Omega )$, we consider the quasi-linear operator $Q:C^2(\Omega )\to R(\Omega )$ given by,
\begin{equation} \label{pi1''} 
Q[u]:= \sum_{i,j=1}^n A_{ij}(x,u,Du)u_{x_ix_j} + B(x,u,Du) \ \ \ \text{in } \Omega ,
\end{equation} 
with $A_{ij},B:\Omega\times \mathbb{R}\times\mathbb{R}^n \to\mathbb{R}$ prescribed functions.  
Specifically, we refer to $Q$ as elliptic with respect to a specific $u\in C^2(\Omega )$ if \eqref{nc2} holds for $a_{ij}(x) = A_{ij}(x,u(x),Du(x))$ for all $x\in\Omega$. 

\subsection{CSMP and BPL for linear elliptic PDI}
Before, we establish a CSMP and BPL for classical solutions to $L[u]\geq 0$ with $L$ given by \eqref{nc1}, we give the following lemma which guarantees the existence of a suitable {\emph{comparison function}}. 

\begin{lem} \label{lem1}
Let $R,m>0$ be constants and set $k=2n\left( \tfrac{2}{R} + 1\right) + 3$. 
Additionally, suppose that there exists a constant $\epsilon >0$ and a continuous non-increasing function $\Lambda :(0,\epsilon ]\to (0,\infty )$ such that $\Lambda \in L^1((0,\epsilon ])$,  
\begin{equation} \label{lem1p3''} 
\epsilon \in \left( 0, \min \left\{ 1,\ \tfrac{R}{2} \right\}\right)\ \text{ and }\ \int_0^\epsilon \Lambda (s) ds < \frac{1}{k} ,
\end{equation}
and moreover, for $\Omega = B_R(O)\setminus \bar{B}_{R-\epsilon}(O)$ that
\begin{equation} \label{lem1H1} 
|y|^2 \leq \sum_{i,j=1}^n a_{ij}(x)y_iy_j \leq \Lambda (R-|x|)|y|^2 \ \ \ \forall x\in \Omega ,\ y\in \mathbb{R}^n ,
\end{equation}
\begin{equation} \label{lem1H2}  
|b_i(x)| \leq \Lambda (R-|x|) \ \ \ \ \forall x\in \Omega ,
\end{equation}
\begin{equation} \label{lem1H3}  
-c(x) \leq \frac{\Lambda (R-|x|)}{(R-|x|)} \ \ \ \ \forall x\in \Omega .
\end{equation}
Then, if $L$ is a linear elliptic operator with coefficients that satisfy \eqref{lem1H1}-\eqref{lem1H3}, there exists $v:\bar{\Omega} \to [0,m]$, such that:
\begin{itemize}
\item[(i)] $v=0$ on $\partial B_R(O)$, $v=m$ on $\partial B_{R-\epsilon}(O)$, and $v>0$ on $\Omega$.
\item[(ii)] $v\in C^1(\bar{\Omega}) \cap C^2(\Omega )$.
\item[(iii)] $L[v]>0$ on $\Omega$. 
\item[(iv)] $\partial_\nu v < 0$ on $\partial B_{R}(O)$ where $\partial_\nu v$ denotes the outward (to $\Omega$) directional derivative of $v$ normal to $\partial \Omega$.
\end{itemize}
\end{lem} 

\begin{proof}
Define $f:[0,\epsilon ]\to [0,\infty )$ to be 
\begin{equation} \label{lem1p1} 
f(r) = r + k\int_0^r\hspace{-6pt}\int_0^s \Lambda (t) dt ds \ \ \ \forall r\in [0,\epsilon ] . 
\end{equation}
It follows immediately that 
\begin{equation} \label{lem1p2'} 
f\in C^1([0,\epsilon ])\cap C^2((0,\epsilon ]), 
\end{equation}
with
\begin{equation} \label{lem1p2} 
f'(r) = 1+ k\int_0^r \Lambda (t) dt \ \ \ \forall r\in [0,\epsilon ], 
\end{equation}
\begin{equation} \label{lem1p3} 
f''(r) = k\Lambda (r) \ \ \ \forall r\in (0,\epsilon ]. 
\end{equation}
Now, we define $\tilde{v}:\bar\Omega\to\mathbb{R}$ to be 
\begin{equation} \label{lem1p4} 
\tilde{v} (x) = f(R-|x|) \ \ \ \forall x\in \bar\Omega . 
\end{equation}
It follows from \eqref{lem1p1}-\eqref{lem1p4} that 
\begin{equation} \label{lem1p5'} 
\tilde{v}\in C^1(\bar{\Omega} )\cap C^2(\Omega ) , 
\end{equation}
\begin{equation} \label{v9.14'} 
\tilde{v} >0 \text{ on } \Omega  , 
\end{equation}
and 
\begin{align}
\nonumber L[\tilde{v}](x) = & \frac{\left( f''(R-|x|)|x|+ f'(R-|x|)\right)}{|x|^3}\sum_{i,j=1}^n a_{ij}(x)x_ix_j  \\
\label{lem1p7} & - \frac{f'(R-|x|)}{|x|} \sum_{i=1}^n \left( a_{ii}(x) + b_i(x)x_i \right) + f(R-|x|)c(x)
\end{align}
for all $x\in\Omega$. 
It now follows from substituting \eqref{lem1p1}, \eqref{lem1p2} and \eqref{lem1p3} into \eqref{lem1p7}, and using \eqref{lem1p3''}-\eqref{lem1H3}, that 
\begin{align}
\nonumber L[\tilde{v}](x) & = \frac{1}{|x|^3}\left( k\Lambda (R-|x|)|x| + \left(1+k\int_0^{R-|x|} \Lambda (t) dt\right) \right) \sum_{i,j=1}^n a_{ij}(x) x_ix_j \\
\nonumber & \ \ \ - \frac{1}{|x|}\left(1+k\int_0^{R-|x|} \Lambda (t) dt\right) \sum_{i=1}^n (a_{ii}(x) + b_i(x)x_i) \\
\nonumber & \ \ \ + c(x)\left( (R-|x|) + k\int_0^{R-|x| }\hspace{-6pt}\int_0^s \Lambda (t) dt ds \right) \\
\nonumber & \geq \Lambda (R-|x|) \left( k - 2n\left(\frac{2}{R} + 1 \right) - 2 \right) \\
\nonumber & = \Lambda (R-|x|) \\
\label{lem1p8} &>0
\end{align}
for all $x\in \Omega$. 
Now, define $v:\bar\Omega\to\mathbb{R}$ to be 
\begin{equation} \label{lem1p12} 
v(x) = \frac{\tilde{v}(x)m}{f(\epsilon )} \ \ \ \forall x\in \bar\Omega . 
\end{equation}
Then, via \eqref{lem1p12}, \eqref{lem1p4}, \eqref{v9.14'} and \eqref{lem1p1}, $v$ satisfies (i). 
Also, via \eqref{lem1p5'} $v$ satisfies (ii). 
Additionally, from \eqref{lem1p8} and \eqref{lem1p12}, $v$ satisfies (iii). 
Moreover, via \eqref{lem1p12}, \eqref{lem1p4}, \eqref{lem1p2} and \eqref{lem1p1}, it follows that 
\[ 
\partial_\nu v (x)|_{|x|=R} = \frac{-m}{f(\epsilon )} <0 ,
\]
and hence $v$ satisfies (iv), as required.  
\end{proof}

We now establish a CSMP for linear elliptic PDI which allows coefficients of $L$ to blow-up in neighborhoods of interior points of $\Omega$.
We note that one can recover a standard CSMP for linear elliptic PDI with bounded coefficients of appropriate sign (see for instance \cite{EH1}, \cite{PW1} or \cite{PPJS1}) by considering $\mathcal{S}=\emptyset$ with $\lambda$ a sufficiently large constant.

\begin{thm}[CSMP] \label{CSMP}
Let $\Omega\subset \mathbb{R}^n$ and $\mathcal{S}\subset\Omega$ satisfy the outward ball property.
Suppose that $u\in C^2(\Omega )$ satisfies the linear elliptic PDI $L[u]\geq 0$ on $\Omega \setminus \mathcal{S}$. 
In addition, suppose that for each $B_R(x_0)\subset (\Omega \setminus \mathcal{S})$ for which $\partial B_R(x_0) \cap \partial \Omega = \emptyset$, there exists a function $\Lambda :\left( 0,\frac{R}{2}\right] \to (0,\infty )$ which is continuous non-increasing and such that $\Lambda\in L^1\left(\left( 0,\tfrac{R}{2}\right]\right)$, and such that the coefficients of $L$ satisfy
\begin{equation} \label{csmp1} 
|y|^2 \leq \sum_{i,j=1}^n a_{ij}(x)y_iy_j \leq \Lambda (d(\{ x\} ,\partial B_R(x_0))) |y|^2 \ \ \ \forall x\in B_R(x_0) ,\ y\in \mathbb{R}^n , 
\end{equation}
\begin{equation} \label{csmp2} 
|b_i(x) | \leq \Lambda (d(\{ x \} ,\partial B_R(x_0))) \ \ \ \forall x\in B_R(x_0) , 
\end{equation}
\begin{equation} \label{csmp3} 
c(x) \geq \frac{-\Lambda (d(\{ x\} ,\partial B_R(x_0))) }{d(\{ x\} ,\partial B_R(x_0))} \ \ \ \forall x\in B_R(x_0) . 
\end{equation}
Additionally, let
\begin{equation} \label{csmp3'}
M_u=\sup_{x\in\Omega}u(x) 
\end{equation}
and suppose that either $M_u=0$, or $M_u >0$ with $c$ non-positive.
Then, $M_u > u(x)$ for all $x\in \Omega$ or $u$ is constant on $\Omega$. 
\end{thm}

\begin{proof}
Suppose that $u$ is not constant on $\Omega$, and 
\begin{equation} \label{CSMP1} 
\mathcal{T} = \left\{ x\in \Omega :\ u(x) = M_u \right\} 
\end{equation}
is not empty. 
Since $\mathcal{T}$ is a relatively closed strict subset of $\Omega$ and $\mathcal{S}$ satisfies the outward ball property, it follows that there exists a sufficiently small $B_R(x_0)$ such that $B_R(x_0)\subset \Omega \setminus (\mathcal{T}\cup \mathcal{S})$, $\partial B_R(x_0) \cap \mathcal{T} = \{ y_0 \}$ and $\partial B_R(x_0) \cap \partial \Omega = \emptyset$. 
Moreover, it follows from \eqref{csmp1}-\eqref{csmp3} and the hypotheses on $L$ and $\Lambda$, that Lemma \ref{lem1} can be applied to a linear elliptic operator $\tilde{L}$ defined in $\Omega_0 = B_{R}(O) \setminus \bar{B}_{R-\epsilon}(O)$, for sufficiently small $\epsilon \in (0,R)$, with coefficients given by
\begin{equation} \label{edit1} 
\tilde{a}_{ij}(x)=a_{ij}(x+x_0),\ \ \ \tilde{b}_i(x) = b_i(x+x_0),\ \ \  \tilde{c}(x) = \frac{-\Lambda (R-|x|)}{(R-|x|)},\ \ \ \forall x\in \Omega_0 
\end{equation}
with 
\[ m=M_u-\left( \sup_{\partial B_{R-\epsilon}(x_0)} u \right) >0,\]
to guarantee the existence of $v:\bar{B}_{R}(O)\setminus B_{R-\epsilon}(O)\to [0,m]$ that satisfies the conclusions of Lemma \ref{lem1}.
Now, define $w:\overline{\Omega}_0\to\mathbb{R}$ to be 
\begin{equation} \label{bplp1} 
w(x) = u(x+x_0) + v(x) - M_u \ \ \ \forall x\in \overline{\Omega}_0 .
\end{equation}
It follows that $w\in C^2(\Omega_0)\cap C^1(\overline{\Omega}_0)$ and $\sup_{\partial \Omega_0}w = w(y_0-x_0)=0$.
Additionally, it follows that $w \leq 0$ on $\Omega_0$, for suppose that the converse holds i.e. that there exists $x^*\in \Omega_0$ such that $\sup_{x\in\Omega_0}w(x) = w(x^*)>0$. 
Then since $L[u](x^*+x_0)\geq 0$ and $\tilde{L}[v](x^*) >0$, via \eqref{edit1} and \eqref{bplp1} we have,
\begin{align}
\label{edit3} \sum_{i,j=1}^n \tilde{a}_{ij}(x^*)w_{x_ix_j}(x^*) + \sum_{i=1}^n \tilde{b}_i(x^*) w_{x_i}(x^*) & > -c(x^*+x_0)u(x^*+x_0)-\tilde{c}(x^*)v(x^*)\\
\nonumber & \geq \frac{\Lambda (R-|x^*|)}{(R-|x^*|)}(\min\{ u(x^*+x_0),\ 0\} + v(x^*)) \\
\nonumber & > 0
\end{align}
via \eqref{csmp3}, \eqref{bplp1} and the hypotheses. 
However, since there is a local maxima of $w$ at $x^*$, then $Dw (x^*) =0$, and $D^2w (x^*)$ is negative semi-definite.
Consequently, via the Schur Product Theorem, the left hand side of \eqref{edit3} is non-positive, which gives a contradiction, and hence, $w \leq 0$ on $\Omega_0$.
Therefore, $\partial_\nu w(y_0-x_0) \geq 0$, and hence $\partial_\nu u(y_0) \geq -\partial_\nu v(y_0-x_0) >0$. 
However, since there is a local maxima of $u$ at $y_0$, it follows from the regularity of $u$ that $Du(y_0)=0$, which contradicts $\partial_\nu u(y_0) >0$.
Therefore, either $u$ is constant on $\Omega$, or $u<M_u$ on $\Omega$, as required. 
\end{proof}

\begin{remk}
Note that in Theorem \ref{CSMP}, the conditions on the coefficients of $L$ apply on balls which satisfy $\partial B_{R}(x_0)\cap \bar{\mathcal{S}}\not= \emptyset$ but not on balls which satisfy $B_{R}(x_0)\cap \bar{\mathcal{S}}\not= \emptyset$.
Thus, although the coefficients of $L$ can, under constraints \eqref{csmp1}-\eqref{csmp3}, blow-up as $x\to \bar{\mathcal{S}}$, they cannot blow up (except $c$ negatively) as $x\to x_0$ for $x_0\in \Omega \setminus \bar{\mathcal{S}}$.
Moreover, observe that the coefficients of $L$ can blow-up as $x\to\partial \Omega$ with conditions \eqref{csmp1}-\eqref{csmp3} not required to hold on $B_R(x_0)$ such that $\partial B_R(x_0)\cap \partial \Omega \not= \emptyset$. 
However, for a BPL to hold for a linear elliptic operator $L$ on $\Omega$, conditions \eqref{csmp1}-\eqref{csmp3} are required to hold on balls $B_R(x_0)$ such that $\partial B_R(x_0)\cap \partial \Omega \not= \emptyset$.
This is the principal difference in hypothesis between BPL and CSMP for linear elliptic PDI.
\end{remk}

A straightforward application of Theorem \ref{CSMP} gives an associated BPL for classical solutions to linear elliptic PDI. 

\begin{thm}[BPL] \label{BPL1}
Suppose that the hypotheses of Theorem \ref{CSMP} hold, with the restriction that `for which $\partial B_{R}(x_0) \cap \partial \Omega = \emptyset$' is omitted.\footnote{i.e. for the BPL, we also impose conditions \eqref{csmp1}-\eqref{csmp3} on $B_R(x_0)\subset \Omega \setminus \mathcal{S}$ such that $\partial B_R(x_0)\cap \partial \Omega \not= \emptyset$.} 
In addition, suppose that $u\in C^1(\bar{\Omega})$ and $\sup_{x\in\Omega} u(x)=u(x_b)$ for some $x_b\in \partial \Omega$ such that there exists $B_{R_b}(x_b')\subset\Omega \setminus\mathcal{S}$ that satisfies $x_b \in \partial B_{R_b}(x_b')$.
If $u$ is not constant on $\Omega$, then $\partial_\nu u(x_b) >0$.
\end{thm}

\begin{proof}
Since $u$ satisfies the conditions of Theorem \ref{CSMP} and is not constant, it follows that $u(x) < u(x_b)$ for all $x\in B_{R_b}(x_b')$. 
A function analogous to $w$ in \eqref{bplp1} can now be constructed, from which, we can conclude (as in the proof of Theorem \ref{CSMP}) that $\partial_\nu u(x_b)>0$, as required.    
\end{proof}

\subsection{Comparison-type BPL for elliptic classical solutions to quasi-linear PDI}
In this subsection we establish a comparison type BPL for classical elliptic solutions to quasi-linear PDI using the approach described in \cite[Chapter 2]{PPJS1}. 
Specifically, via an application of Theorem \ref{BPL1}, a BPL for classical elliptic solutions to quasi-linear PDI can be established. 
Although the proof is standard, we provide it to inform the discussion that follows.

\begin{thm}[BPL]\label{t2} 
Suppose that $u,v:\bar\Omega\to\mathbb{R}$ satisfy $u,v\in C^2(\Omega )\cap C^1(\bar{\Omega })$ and the quasi-linear PDI $Q[u]\geq 0$ and $Q[v]\leq 0$ on $\Omega$. 
Furthermore, suppose that $Q$ is elliptic with respect to $u$, with $v_{x_ix_j}$ bounded on $\Omega$ (or instead suppose that $Q$ is elliptic with respect to $v$, with $u_{x_ix_j}$ bounded on $\Omega$) for $i,j=1,\ldots,n$.
Suppose that $u<v$ in $\Omega$ and $u=v$ at $x_b\in\partial \Omega$ for which, there exists $B_{R_b}(x_b')\in\Omega$ with $x_b \in \partial B_{R_b}(x_b')$.
Suppose that there exists a continuous non-increasing function $\Lambda :\left( 0, \frac{R_b}{2} \right] \to (0,\infty )$ such that $\Lambda \in L^1\left( \left( 0, \frac{R_b}{2} \right]\right)$,  
\begin{equation} \label{t2A}
|A_{ij}(x,z_1,\eta_1) - A_{ij}(x,z_2,\eta_2)| \leq \Lambda (d(\{ x\} ,\partial B_{R_b}(x_b') )) \left( \frac{|z_1-z_2|}{d(\{ x\} ,\partial B_{R_b}(x_b'))} + \sum_{l=1}^n|\eta_{1l}-\eta_{2l}| \right) 
\end{equation}
for all $(x,z_1,\eta_1),(x,z_2,\eta_2)\in B_{R_b}(x_b') \times [-M_z,M_z]\times [-M_\eta , M_\eta ]^n$, and
\begin{equation} \label{t2B} 
B(x,z_1 , \eta_1 ) - B(x,z_2,\eta_2 ) \geq - \Lambda (d(\{ x\} ,\partial B_{R_b}(x_b') )) \left( \frac{(z_1 - z_2)}{d(\{ x\} , \partial B_{R_b}(x_b') )} + \sum_{l=1}^n|\eta_{1l}-\eta_{2l}| \right)  
\end{equation}
for all $(x,z_1,\eta_1),(x,z_2,\eta_2)\in B_{R_b}(x_b') \times [-M_z,M_z]\times [-M_\eta , M_\eta ]^n$ with $z_1 \geq z_2$, with 
\[ 
M_z=\sup_{x\in B_{R_b}(x_b')} \{ |u(x)|,|v(x)|\} \text{ and } M_\eta = \sup_{\substack{x\in B_{R_b}(x_b') \\ i=1,\dots ,n}}\{ |u_{x_i}(x)|,|v_{x_i}(x)|\} .
\]
Then $\partial_\nu u(x_b) > \partial_\nu v(x_b)$.
\end{thm}

\begin{proof}
Let $w=u-v$ on $\overline{B_{R_b}(x_b')}$. 
Then, on $B_{R_b}(x_b')$, 
\begin{align*}
\nonumber 0 &\leq \sum_{i,j=1}^n\left( A_{ij} (\cdot ,u,Du) u_{x_ix_j} - A_{ij} (\cdot ,v,Dv) v_{x_ix_j}\right) + B(\cdot ,u,Du) - B(\cdot ,v,Dv) \\
\nonumber & = \sum_{i,j=1}^n \big(A_{ij}(\cdot ,u,Du)(u_{x_ix_j} - v_{x_ix_j})+ (A_{ij}(\cdot ,u,Du) -A_{ij}(\cdot ,u,Dv)) v_{x_ix_j} \\
\nonumber & \hspace{10mm} + \left( A_{ij}(\cdot ,u,Dv) -A_{ij}(\cdot ,v,Dv)\right) v_{x_ix_j}\big) \\  
\nonumber & \ \ \  + \left( B(\cdot ,u,Du) - B(\cdot ,u,Dv)\right) + \left( B(\cdot ,u, Dv) - B(\cdot , v, Dv)\right) \\
\nonumber & \leq \sum_{i,j=1}^n A_{ij}(\cdot , u , Du) w_{x_ix_j} + \Lambda (d(\cdot , \partial B_{R_b}(x_b') )) \sum_{i=1}^n \left( \text{sgn} (w_{x_i}) + \sum_{k,l=1}^n v_{x_kx_l}\text{sgn} (v_{x_kx_l}w_{x_i}) \right)w_{x_i} \\
\nonumber & \ \ \ + \left( \frac{\Lambda (d(\cdot , \partial B_{R_b}(x_b') ))}{d(\cdot , \partial B_{R_b}(x_b') )} \sum_{k,l=1}^n v_{x_kx_l}\text{sgn} (v_{x_kx_l}w) + \frac{\left( B(\cdot ,u, Dv) - B(\cdot , v, Dv)\right)}{w}\right) w \\
\nonumber & =: \sum_{i,j=1}^n \tilde{a}w_{x_ix_j} + \sum_{i=1}^n \tilde{b}_kw_{x_k} + \tilde{c}w   
\end{align*}
where, for $i,j=1,\dots ,n$, $\tilde{a}_{ij},\tilde{b}_i,\tilde{c}: \Omega\to\mathbb{R}$ are given by, 
\begin{align}
\tilde{a}_{ij}=&\ A_{ij}(\cdot,u,Du) \label{a} , \\
\tilde{b}_i=& \Lambda (d(\cdot ,\partial B_{R_b}(x_b') ))\left( \text{sgn}(w_{x_i}) +  \sum_{k,l=1}^n  v_{x_kx_l}\text{sgn}(v_{x_kx_l}w_{x_i})\right) \label{b} ,\\
\tilde{c} =& \frac{\Lambda (d(\cdot ,\partial B_{R_b}(x_b') ))}{d(\cdot ,\partial B_{R_b}(x_b') )}\sum_{k,l=1}^n v_{x_kx_l} \text{sgn}(v_{x_kx_l}w) + \frac{B(\cdot,u,Dv)-B(\cdot,v,Dv)}{(u-v)} \label{c} \\
\geq & -\frac{\Lambda (d(\cdot ,\partial B_{R_b}(x_b') ))}{d(\cdot ,\partial B_{R_b}(x_b') )}(n^2\sup_{\substack{k,l=1,...,n \\ x\in B_{R_b}(x_b')}}|v_{x_kx_l}(x)| +1)  , \label{d}
\end{align}
on $B_{R_b}(x_b')$.
Thus, it follows that $\tilde{L}$ is a linear elliptic operator on $B_{R_b}(x_b')$, that satisfies the conditions of Theorem \ref{BPL1}, provided that we consider $\Lambda$ in Theorem \ref{BPL1} as that in \eqref{a}-\eqref{d} after multiplication by a sufficiently large constant. 
An application of Theorem \ref{BPL1} yields $\partial_\nu w >0$ at $x_b$ and hence,
\[ 
\partial_\nu u(x_b) > \partial_\nu v (x_b), 
\]
as required.
\end{proof}

\begin{remk}
Note that conditions \eqref{t2A} and \eqref{t2B} ensure that: 
$A_{ij}$ are locally Lipschitz continuous in $z$ and $\eta$ on $\Omega$; 
$B$ is locally lower Lipschitz in $z$ and Lipschitz continuous in $\eta$;
and the associated Lipschitz and lower Lipschitz constants for $A_{ij}$ and $B$ can tend to $\infty$ as $x\to \partial \Omega $ but are constrained by the integrability condition on $\Lambda$.
Additionally observe that the conditions in Theorem \ref{t2} can be readily altered to accommodate $d(x,\partial \Omega)$ instead of $d(x ,\partial B_{R_b}(x_b'))$.
\end{remk}
   
We now demonstrate that if the bound on the lower Lipschitz constant for $B$ in Theorem \ref{t2} is relaxed to a mere local lower Lipschitz condition, then the conclusion of Theorem \ref{t2} does not necessarily hold.   

\begin{ex} \label{exam1}
Suppose that $\Omega\subset\mathbb{R}^n$ and for $x_b\in\partial \Omega$ there exists $B_{R_b}(x_b')\in\Omega$ with $x_b \in \partial B_{R_b}(x_b')$.
Consider $u:\bar{\Omega}\to\mathbb{R}$ given by,  
\begin{equation} \label{e1} 
u(x) = 0 \ \ \ \forall x\in \bar{\Omega}, 
\end{equation}
and $v:\bar{\Omega}\to\mathbb{R}$ such that:
\begin{equation} \label{xv1} 
v\in C^\infty(\bar{\Omega}),
\end{equation}
\begin{equation} \label{xv2} 
v>0 \text{ in } \Omega , 
\end{equation}
\begin{equation} \label{xv3} 
v(x_b)=0 \text{ and } \partial_\nu v(x_b)=0. 
\end{equation}
Note that \eqref{xv1} implies that there exists $M\geq 0$ such that for $i,j=1,\ldots , n$,
\begin{equation} \label{31'}
|v|,\ |v_{x_i}|,\ |v_{x_ix_j}| \leq M \text{ for all }x\in \bar\Omega . 
\end{equation}
Now, for the quasi-linear PDI in \eqref{pi1''}, set $A_{ij}:\Omega\times\mathbb{R}\times\mathbb{R}^n\to\mathbb{R}$ to be 
\begin{equation} \label{e2} 
A_{ij}(x,z,\eta ) = \delta_{ij} \ \ \ \forall (x,z,\eta )\in \Omega\times\mathbb{R}\times\mathbb{R}^n , 
\end{equation}
for all $i,j=1,\dots ,n$ and $B:\Omega\times\mathbb{R}\times\mathbb{R}^n\to\mathbb{R}$ to be
\begin{equation} \label{e3} 
B(x,z,\eta ) = -\frac{z}{v(x)}\sum_{i=1}^nv_{x_ix_i}(x) \ \ \ \forall (x,z,\eta )\in \Omega\times\mathbb{R}\times\mathbb{R}^n . 
\end{equation}
Since the coefficients of $A_{ij}$ define a Laplacian, it can be seen that $A_{ij}$ satisfies the conditions of Theorem \ref{t2} (with, for example, $\Lambda = 1$), and also, that $Q$ is elliptic with respect to $u$ with $v_{x_ix_j}$ bounded on $\Omega$. 
Moreover, observe that $B$ is independent of $\eta$, and
\begin{equation} \label{e5} 
B(x,z_1 ,\eta_1 )-B(x, z_2 ,\eta_2 ) = -\frac{z_1-z_2}{v(x)}\sum_{i=1}^nv_{x_ix_i}(x) \geq -M_K(z_1- z_2) ,  
\end{equation}
for all $(x,z_1 ,\eta_1 ),(x, z_2 ,\eta_2)\in K$ such that $z_1 \geq z_2 $, with $K= \Omega ' \times [-M,M]\times [-M,M]^n$ for any $\Omega '$ that is a compact subset of $\Omega$, where via \eqref{xv2} and \eqref{31'},
\begin{equation} \label{36'}  
M_K=\frac{Mn}{\inf_{x\in\Omega '}\{v(x)\}} > 0 . 
\end{equation}
Therefore, it follows from \eqref{e5} and \eqref{36'} that $B(x,z,\eta )$ satisfies the conditions of Theorem \ref{t2} with the exception of the lower Lipschitz condition, which instead holds only locally on $\Omega \times \mathbb{R}\times\mathbb{R}^n$. 
It follows from \eqref{e1}-\eqref{xv3}, that $Q[u]\geq 0$ and $Q[v]\leq 0$ on $\Omega$, for $Q$ defined by \eqref{e2} and \eqref{e3}. 
Moreover, via \eqref{e1}, \eqref{xv2} and \eqref{xv3} $u<v$ in $\Omega$ and $u(x_b)=v(x_b)$ for $x_b\in\partial \Omega$. 
In conclusion, although $u$, $v$, $A_{ij}$ and $B$ satisfy all of the conditions of Theorem \ref{t2} (with the exception of the lower Lipschitz condition on $B$, or alternatively \eqref{t2B}), via \eqref{xv3}, 
\[ \partial_\nu u(x_b) = \partial_\nu v(x_b) , \]
which violates the conclusion of Theorem \ref{t2}. 
\end{ex}


\begin{remk}
We note that $u$, $v$, $A_{ij}$ and $B$ in Example \ref{exam1} satisfy all of the conditions of \cite[Theorem 2.7.1]{PPJS1}, but violate the conclusion. 
This occurs since an unconstrained local lower Lipschitz constant is supposed on $B$ with respect to $z$ in \cite[Theorem 2.7.1]{PPJS1}, which is an error. 
It is noteworthy that essentially the same error can be found in the statement of a BPL for classical solutions to linear parabolic PDI given in \cite[p.174, Theorem 7]{PW1}, as illustrated in \cite{DNJM1}.
We also highlight that in both of these instances, a direct proof of the associated BPL is not given, but instead, only the main ideas of the proofs are described. 
\end{remk}

\begin{remk}
If $\partial_{\nu\nu} v(x_b) >0$ in Example \ref{exam1}, by considering $K=B_R(x_b')\times [-M,M]\times [-M,M]^n$ with $0<R<R_b$, it follows from \eqref{xv1} that as $R\to R_b$,
\begin{equation} \label{rem1} 
M_K \geq \frac{2Mn}{v_{\nu\nu}(x_b) ( R_b-R )^2 + O( (R-R_b)^3)} \geq \frac{Mn}{v_{\nu\nu}(x_b)d(\partial B_{R}(x_b'), \partial B_{R_b}(x_b'))^2} .
\end{equation}
Thus, we observe that $B$ in \eqref{e5} satisfies the conditions Theorem \ref{t2} with the exception of $\Lambda\in L^1\left( \left( 0, \frac{R_b}{2} \right] \right)$ in \eqref{t2B}.
This follow from letting $R\to R_b$ in \eqref{rem1} which implies that $\lambda$ necessarily satisfies 
\[ \Lambda (d) \geq \frac{Mn}{v_{\nu\nu}(x_b)d} \text{ as } d\to 0^+ .\]
\end{remk}

Now, we highlight the necessity of the bound on $v_{x_ix_j}$ (or $u_{x_ix_j}$) in Theorem \ref{t2}.
Note that this condition is not present in \cite[Theorem 2.7.1]{PPJS1}.

\begin{ex}
Let $\Omega = (0,1)\subset \mathbb{R}$ and $u,v:\bar{\Omega}\to\mathbb{R}$ be given by
\begin{equation} \label{1ex3} 
u(x) = x^{1+\alpha },\ \ \ v(x)=2x^{1+\alpha} \ \ \ \forall x\in \bar\Omega , 
\end{equation}
with constant $\alpha \in (0,1)$. 
It follows that $u,v\in C^1(\bar{\Omega} )\cap C^2(\Omega )$, $v>u$ in $\Omega$, and for $x_b=0\in\partial \Omega$, we have $u(x_b)=v(x_b)=0$.
Now, consider the quasi-linear operator $Q$ with $A:\Omega \times \mathbb{R}\times \mathbb{R} \to\mathbb{R}$ given by
\begin{equation} \label{5ex3} 
A(x,z,\eta ) = 1 + \frac{2}{x^{1+\alpha}}\left(\frac{3}{2}x^{1+\alpha} - z\right) \ \ \  \forall (x,z,\eta )\in \Omega \times \mathbb{R}\times \mathbb{R} 
\end{equation}
with $B=0$ on $\Omega \times \mathbb{R}\times \mathbb{R} \to\mathbb{R}$. 
Since 
\begin{equation} \label{ed2} 
Q[u]= A(\cdot,u,Du)u_{xx} = 2u_{xx}\geq 0,\ \ \ Q[v]= A(\cdot,v,Dv)v_{xx} \leq 0 ,
\end{equation}
on $\Omega$, it follows that $Q$ is elliptic with respect to $u$, and that $Q$ satisfies the conditions \eqref{t2A} and \eqref{t2B} in Theorem \ref{t2} with $\Lambda :(0,1]\to (0,\infty )$ given by 
\[ 
\Lambda (d) = \frac{2}{d^{\alpha }} \ \ \ \forall d \in \left( 0,\tfrac{1}{2}\right] . 
\]
Since $\Lambda$ is continuous non-increasing and $\Lambda \in L^1\left( \left( 0,\tfrac{1}{2}\right] \right)$, it follows from \eqref{1ex3}-\eqref{ed2} that $u$ and $v$ satisfy all of the conditions of Theorem \ref{t2} with the exception of $v_{xx}$ being bounded on $\Omega$.
However, via \eqref{1ex3},
\[ 
u_{\nu}(x_b)=v_\nu(x_b)=0 ,
\]
which violates the conclusion of Theorem \ref{t2}. 
\end{ex}

\section{Weak Theory}

In this section, we establish a comparison-type tangency principle, for $C^1$ weak elliptic solutions to divergence structure PDI which is a correction of that stated in \cite[Theorem 2.7.2]{PPJS1}.
The proof largely follows that of \cite[Theorem 2.7.2]{PPJS1} with additional details included to highlight the additional hypotheses.
We also provide simple counter-examples to \cite[Theorems 2.7.2 and 2.7.3]{PPJS1}.

\subsection{Notation and Definitions}

The quasi-linear divergence structure PDI we consider are given by:
\begin{equation} \label{5a} 
\text{div}(A(\cdot , u, Du)) + B(\cdot , u, Du) \geq 0 \ \ \ \text{ on } \Omega ,
\end{equation}
\begin{equation} \label{5b} 
\text{div}(A(\cdot , v, Dv)) + B(\cdot , v, Dv) \leq 0 \ \ \ \text{ on } \Omega ,
\end{equation}
with $A:\Omega \times \mathbb{R} \times \mathbb{R}^n\to\mathbb{R}^n$ and $B:\Omega \times \mathbb{R} \times \mathbb{R}^n\to\mathbb{R}$.
Specifically, we consider $C^1$ weak solutions to \eqref{5a} (and analogously \eqref{5b}) that satisfy: 
$u\in C^1(\bar{\Omega})$, $A(\cdot , u, Du),\ B(\cdot , u, Du)\in L_{\textrm{loc}}^1(\Omega )$ and 
\begin{equation} \label{5a'}
 \int_{\Omega} A (x,u(x),Du(x)) \cdot D\psi (x) dx \leq \int_\Omega B(x,u(x),Du(x)) \psi(x) dx
\end{equation}
for any test function $\psi\in C^1(\bar{\Omega})$ such that $\psi \geq 0$ on $\Omega$ and $\psi$ has compact support in $\Omega$.
Moreover, we say that $u$ (and analogously $v$) is an elliptic solution to \eqref{5a} if $a_{ij}(x)=(A_i)_{\eta_j} (x,u(x),Du(x))$ satisfies the left inequality in \eqref{nc2} for all $x\in\Omega$. 
Furthermore, in this section we consider $\Omega$ with boundary $\partial \Omega$ that satisfies an {\emph{interior cone condition}} i.e. at each point $x_b\in\partial \Omega$ there exists a cone of finite height in $\Omega$ with apex $x_b$.
We denote the interior of such a cone by $\Omega_b$.

\subsection{A comparison-type tangency principle for weak elliptic solutions to quasi-linear divergence structure PDI}

\begin{thm}[Tangency Principle] \label{tfinal}
Let $x_b\in \partial \Omega$ satisfy the interior cone condition, and $u,v:\bar\Omega\to \mathbb{R}$ be such that: 
$u,v\in C^1(\bar\Omega )$; 
$u,v$ satisfy \eqref{5a} and \eqref{5b} respectively; 
$A:\Omega \times \mathbb{R}\times \mathbb{R}^n\to\mathbb{R}^n$ is continuous and continuously differentiable with respect to $z$ and $\eta$;
$A_z$ is uniformly bounded and $A_\eta$ is uniformly continuous on $\Omega_b\times [u(x_b) - M_z , u(x_b) + M_z ]\times [-M_\eta , M_\eta ]^n$ for some constants $M_z , M_\eta >0 $;
$B:\Omega \times \mathbb{R}\times \mathbb{R}^n\to\mathbb{R}$ satisfies 
\begin{equation} \label{Bcond}
B(x,z_1 , \eta_1 ) - B(x,z_2,\eta_2 ) \geq -b_z(z_1 - z_2) - b_\eta \sum_{l=1}^n|\eta_{1l}-\eta_{2l}|  
\end{equation}
for all $(x,z_1,\eta_1),(x,z_2,\eta_2)\in \Omega_b\times [u(x_b) - M_z , u(x_b) + M_z ]\times [-M_\eta , M_\eta ]^n$ with $z_1 \geq z_2$ for some constants $b_z, b_\eta \geq 0$;
$u$ is an elliptic solution of \eqref{5a} with respect to $A$ in $\Omega_b$; 
$u<v$ in $\Omega_b$;
and $u(x_b)=v(x_b)$.
Then the zero of $v-u$ at $x_b$ is of finite order.
\end{thm} 

\begin{proof}
For a contradiction, assume that $w=v-u$ has a zero of infinite order at $x_b\in\partial \Omega$.
Via regularity on $w$, it follows that $Dw(x_b)=0$.
Moreover, for each $\epsilon \in (0, \min \left\{ M_z, M_\eta /2 \right\})$, there exists a cone of finite height in $\Omega$ with apex $x_b$, without loss of generality denoted by $\Omega_b$, such that $(w(x),Dw(x))\in (0 ,\epsilon ] \times [-\epsilon , \epsilon]^n$ for all $x\in\Omega_b$, there exists a constant
\begin{equation} \label{5pb}
 a_z =\sup_{\substack{\Omega_b\times [-\epsilon , \epsilon] \times [-\epsilon , \epsilon]^n \\ i=1,\ldots ,n}}  |(A_i)_z| \in [0,\infty ),
\end{equation}
and 
\begin{equation} \label{5pf'}
\left| (A_i)_{\eta_j} (x,u(x),\eta^{(1)}) - (A_i)_{\eta_j} (x,u(x),\eta^{(2)}) \right| < \frac{1}{2n^2}
\end{equation}
for all $(x, \eta^{(1)}), (x, \eta^{(2)})\in \Omega_b\times [-2\epsilon , 2\epsilon]^n$ and $i,j=1,\ldots ,n$.
From \eqref{5a} and \eqref{5b}, we have 
\begin{align}
\nonumber		& \text{div}(A(x,v,Dv)-A(x,u,Du)) + (B(x,v,Dv) - B(x,u,Du)) \\	 
\label{5pf}	& = \text{div}(\tilde{A}(x,w,Dw)) + \tilde{B}(x,w,Dw) \leq 0 
\end{align}
on $\Omega_b$. 
The function $\tilde{A}:\Omega_b\times \mathbb{R} \times \mathbb{R}^n\to\mathbb{R}^n$ arises from repeated application of the mean value theorem in \eqref{5pf}, e.g.
\begin{equation} \label{5ph''}
\tilde{A}_i(x,z,\eta ) = (A_i)_z(x,\tilde{z}^{(i)}(x),Dv(x)) z + \sum_{j=1}^n (A_{i})_{\eta_j}(x,u(x),\tilde{\eta}^{(i)}(x)) \eta_j
\end{equation}
for all $(x,z,\eta )\in \Omega_b \times \mathbb{R} \times \mathbb{R}^n$, $\tilde{z}^{(i)}:\Omega_b\to [0,\epsilon]$ and $\tilde{\eta}^{(i)}:\Omega_b\to [-2\epsilon , 2\epsilon]^n$ for $i=1,\ldots , n$.
Similarly, via \eqref{5pf}, we define $\tilde{B}:\Omega_b\times \mathbb{R}\times \mathbb{R}^n\to\mathbb{R}$ as
\begin{align} 
\nonumber \tilde{B}(x,z,\eta )	& = \left( \frac{B(x,v(x),Dv(x)) - B(x,u(x),Dv(x))}{v(x)-u(x)}\right) z \\
\label{5ph'} 							& \ \ \ + \sum_{i=1}^{n}\left( \frac{B(x,u(x),\eta^{(i)}(x))-B(x,u(x),\eta^{(i+1)}(x))}{Dv_{n+1-i}(x)-Du_{n+1-i}(x)}\right) \eta_{n+1-i},  
\end{align}
for all 
\[ (x,z,\eta )\in \mathcal{R} = \left\{ (x,z,\eta ) \in \Omega_b \times \mathbb{R} \times \mathbb{R}^n: Du_j(x)\not= Dv_j(x) \right\} 
\]
with $\eta^{(i)}: \Omega_b\to\mathbb{R}^n$ given by 
\[ 
\eta_j^{(i)}(x)=
\begin{cases} 
Dv_j(x) , & j \leq n+1-i \\ 
Du_j(x) , & n+2-i \leq j 
\end{cases} 
\]
for $i=1,\ldots ,n+1$.
$\tilde{B}$ is defined analogously on $( \Omega_b \times \mathbb{R} \times \mathbb{R}^n) \setminus \mathcal{R}$.
Since $Q$ is elliptic with respect to $u$, it follows from \eqref{5pb}-\eqref{5ph''} that 
\begin{align}
\nonumber 			\eta^T \cdot \tilde{A}(x,z,\eta ) 	
							& 	= \sum_{i=1}^n (A_i)_z (x, \tilde{z}^{(i)}(x), Dv(x)) z \eta_i + \sum_{i,j=1}^n (A_i)_{\eta_j}(x,u(x), Du(x)) \eta_i\eta_j \\
\nonumber			&	\ \ \ + \sum_{i,j=1}^n \left( (A_i)_{\eta_j}(x,u(x), \tilde{\eta}_j^{(i)}(x)) - (A_i)_{\eta_j}(x,u(x), Du(x))\right) \eta_i\eta_j \\
\nonumber 			& \geq -na_zz|\eta| + |\eta|^2 - \frac{1}{2}|\eta|^2\\
\label{5pc}		& \geq \frac{1}{4} |\eta |^2 - (na_z)^2 z^2 
\end{align}
for $(x,z,\eta )\in \Omega_b \times [0,\infty ) \times \mathbb{R}^n$.
Additionally, via the regularity hypotheses on $A$ and $B$ it follows that there exist constants $a_\eta , b_\eta , b_z \geq 0$ such that 
\begin{equation} \label{5pd}
|\tilde{A}(x,z,\eta )| \leq a_\eta |\eta| + \sqrt{n}a_zz 
\end{equation}
for all $(x,z,\eta)\in \Omega_b \times [0,\infty ) \times \mathbb{R}^n$, and
\begin{equation} \label{5pe}
\tilde{B}(x,z,\eta) \geq - b_\eta |\eta| - b_z z 
\end{equation}
for all $(x,z,\eta)\in \Omega_b \times [0,\infty ) \times \mathbb{R}^n$.
It follows from \eqref{5pf}-\eqref{5pe} that on any $B_r(x)\subset \Omega_b$, that $w$, $\tilde{A}$ and $\tilde{B}$
\footnote{
Note that \cite[Theorem 1.2]{NT1} remains true if $u>0$ in $\Omega$, inequalities (1.2) hold on $\Omega \times [0,\infty ) \times E^n$, and the second and third inequalities in (1.2) for $\alpha = 2$ are replaced by $p\cdot A(x,u,p)\geq a_5|p|^2-a_2u^2$ and $B(x,u,p) \geq - b_1|p| - b_2u$ for constants $a_2,b_1,b_2\geq 0$ and $a_5>0$.
}
satisfy the conditions of Trudinger's weak Harnack inequality \cite[Theorem 1.2]{NT1} with constants required in the hypotheses and conclusion, independent of the ball i.e. there exists a constant $C$ independent of $B_{2r}(x)\in \Omega_b$ such that 
\begin{equation} \label{5pg}
\frac{1}{r^n} \int_{B_{2r}(x)} w \ dx \leq C \min_{B_{r}(x)} w \ \ \ \forall B_{2r}(x)\subset \Omega_b.
\end{equation} 

Now, since $x_b$ is the apex of the cone $\Omega_b\subset \Omega$, it follows that there exists a sequence of balls $\{ B_{r_k}(y_k)\}_{k\in\mathbb{N}_0}$: 
that have boundaries that tangentially intersect $\partial \Omega_b$; 
such that $B_{r_k/3}\left( y_k \right) \subset B_{2r_{k+1}/3}\left( y_{k+1} \right)$ for all $k\in \mathbb{N}_0$; 
for which $y_k\to x_b$ as $k\to\infty$; 
$r_{k+1}<r_k$ for $k\in\mathbb{N}_0$; 
and by denoting $\theta$ to be the half-angular opening of the cone, we can set
\begin{equation} \label{5ph}
\frac{r_{k+1}}{r_k} = \frac{|y_{k+1}-x_b|}{|y_k-x_b|} = \frac{1+(\tfrac{1}{3})\sin{(\theta )}}{1+(\tfrac{2}{3})\sin{(\theta )}} = \kappa \in (0,1) 
\end{equation}
for all $k\in \mathbb{N}_0$.
It follows immediately that 
\begin{equation} \label{5pi}
\min_{B_{r_{k}/3}(y_{k})} w \leq \frac{3^n}{\omega_n r_{k}^n} \int_{B_{r_{k}/3}(y_{k})} w \ dx \leq \frac{3^n}{\omega_n r_{k+1}^n} \int_{B_{2r_{k+1}/3}(y_{k+1})} w\ dx \ \ \ \forall k\in \mathbb{N}_0,
\end{equation}
with $\omega_n$ denoting the volume of a Euclidean unit ball in $\mathbb{R}^n$.
By combining \eqref{5pg} and \eqref{5pi}, we have
\begin{equation} \label{5pj}
\min_{B_{r_{k}/3}\left( y_{k} \right)} w \geq L^k \min_{B_{r_{0}/3}\left( y_{0} \right)} w \ \ \ \forall k\in \mathbb{N},
\end{equation} 
with 
\[
L=\frac{\omega_n}{C3^n}.
\]

Now, via our initial assumption, $w$ has a zero of infinite order at $x_b$ and hence via \eqref{5ph}, for each $m\in \mathbb{N}$ there exists a positive constant $c$ independent of $k$ such that 
\begin{equation} \label{5pk}
w(y_k) \leq c |y_k-x_b|^m =c |y_0-x_b|^m\kappa^{mk} \ \ \ \forall k\in \mathbb{N}_0 .
\end{equation}
Now, via \eqref{5pj} and \eqref{5pk}, it follows that there exists a positive constant $c$ independent of $k$ such that 
\begin{equation} \label{5pl}
L^k \leq c\kappa^{mk} \ \ \ \forall k\in \mathbb{N}_0 .
\end{equation}
Letting $k\to \infty$ in \eqref{5pl} implies that 
\begin{equation} \label{5pm}
\kappa^m\geq L.
\end{equation}
However, via \eqref{5ph} $\kappa^m\to 0$ as $m\to\infty$ and hence for all sufficiently large $m$, it follows that \eqref{5pm} yields a contradiction.
Therefore, the zero of $w$ at $x_b$ is of finite order, as required. 
\end{proof}

\begin{remk}
Observe that via the bounds in \eqref{5pc}-\eqref{5pe}, we have ensured that the constant $C$ in \eqref{5pg} exists independently of the choice of ball in $\Omega_b$.
Alternatively, using the conditions of \cite[Theorem 2.7.2]{PPJS1}, although bounds analogous to \eqref{5pc}-\eqref{5pe} hold on any ball in $\Omega_b$, the same constant $C$ is not necessarily valid for every ball i.e. $C$ is potentially dependent on $k$.
Consequently, in the proof of \cite[Theorem 2.7.2]{PPJS1}, although Theorem \cite[Theorem 1.2]{NT1} can be applied to any ball in $\Omega_b$, as in \eqref{5pg} and \eqref{5pj}, the constant $c$ that arises, as in \eqref{5pl}, is not necessarily independent of $k$, which is the source of the error in the proof.
\end{remk}

\begin{ex} \label{exam3}
Suppose that $\Omega\subset\mathbb{R}^n$ and for $x_b\in\partial \Omega$ satisfies an interior cone condition. 
Consider $u:\bar{\Omega}\to\mathbb{R}$ and $v:\bar{\Omega}\to\mathbb{R}$ as given in Example \ref{exam1} such that additionally, $v$ has a zero of infinite order at $x_b$, i.e.
\begin{equation} \label{exam3a}
\partial_\nu^{(m)} v (x_b) = 0 \ \ \ \forall m\in\mathbb{N}.
\end{equation}  
For the quasi-linear partial differential inequalities in \eqref{5a} and \eqref{5b} set $A:\Omega\times\mathbb{R}\times\mathbb{R}^n\to\mathbb{R}^n$ to be 
\begin{equation} \label{exam3b} 
A(x,z,\eta ) = \eta \ \ \ \forall (x,z,\eta )\in \Omega\times\mathbb{R}\times\mathbb{R}^n , 
\end{equation}
with $B:\Omega\times\mathbb{R}\times\mathbb{R}^n\to\mathbb{R}$ as in \eqref{e3}.
It follows that $A$ satisfies the conditions of Theorem \ref{tfinal}, and also, that $u$ and $v$ are elliptic solutions of \eqref{5a} and \eqref{5b} respectively. 
Via \eqref{e5} and \eqref{36'}, observe that $B$ is independent of $\eta$, and satisfies the conditions of Theorem \ref{tfinal} with the exception of the lower Lipschitz condition in \eqref{Bcond}, which instead holds locally on $\Omega \times \mathbb{R}\times\mathbb{R}^n$.  
Moreover, via \eqref{e1} and \eqref{xv2}, it follows that $u<v$ in $\Omega$. 
In conclusion, although $\Omega$, $u$, $v$, $A$ and $B$ satisfy all of the conditions of Theorem \ref{tfinal} (with the exception of the lower Lipschitz condition on $B$), via \eqref{exam3a}, 
\[ 
\partial_\nu^{(m)} u(x_b) = \partial_\nu^{(m)} v(x_b) \ \ \ \forall m\in\mathbb{N},
\]
which violates the conclusion of Theorem \ref{tfinal}.
We also note here that the conditions on $A_\eta$ and $A_z$ in Theorem \ref{tfinal} cannot be relaxed to those in \cite[Theorem 2.7.2]{PPJS1}, which can be observed via similarly constructed counter-examples.
\end{ex}

\begin{remk}
We note that the erroneous tangency principle stated in \cite[Theorem 2.7.2]{PPJS1} was intended to be a relaxation of that in \cite{JS1} to allow for weaker constraints on the nonlinearities $A$ and $B$ in \eqref{5a} and \eqref{5b} as $x\to \partial \Omega$.
However, the constraint on $B$, appears to arise from the very same condition on $B$ in the erroneous BPL stated in \cite[Theorem 2.7.1]{PPJS1}.
\end{remk}

To conclude the section, we note that in \cite[Theorem 2.7.3]{PPJS1}, a strong maximum principle and tangency principle is stated with the regularity conditions on $u$ and $v$ in \cite[Theorem 2.7.2]{PPJS1} relaxed to $u,v\in C(\bar{\Omega})$ but so that $u$ and $v$ also possess strong derivatives in $L_{\textrm{loc}}^2(\Omega)$.
To compensate for these relaxed regularity conditions on $u$ and $v$, stricter regularity conditions are imposed on $A$ and $B$ which we now demonstrate, are insufficient to establish the conclusion.
This establishes that all three theorems in \cite[Section 2.7]{PPJS1} are erroneous.

\begin{ex} \label{ex4}
For $\epsilon \in (0,1)$ consider $\Omega = B_1(0)\setminus \overline{B_{1-\epsilon}(0)}$ with $u$, $A=A(\eta )$ and $B=B(x,z)$ as in Example \ref{exam3}.
Here consider $v$ given by 
\begin{equation} \label{ex4a}
v(x)=
\begin{cases} 
e^{1/(1-|x|^2)},		& x\in B_1(0)\setminus B_{1-\epsilon}(0) \\
0, 							& x\in \partial B_1(0) .
\end{cases}
\end{equation}
Observe that $u,v\in C^\infty(\bar{\Omega})$ and that the zero of $v-u$ on $\partial B_1(0)$ is of infinite order.
Additionally, note that $A$ is locally bounded on $\mathbb{R}^n$ and $B$ is locally bounded and locally lower Lipschitz on $\Omega \times \mathbb{R}$.
Furthermore, for $i=1,\ldots ,n$, we have
\begin{equation} \label{ex4b}
v_{x_ix_i}(x) = \frac{\left( 4x_i^2 - 8x_i^2(1-|x|^2) - 2(1-|x|^2)^2 \right)v(x) }{(1-|x|^2)^4}  \ \ \ \forall x\in \Omega .
\end{equation}
Via \eqref{ex4a} and \eqref{ex4b}, for sufficiently small $\epsilon >0$, it follows that
\[ 
\sum_{i=1}^n v_{x_ix_i} > 0 \ \ \ \text{ on }\Omega .
\]
For such $\epsilon >0$, it follows that $B(x,z)$, as given by \eqref{e3}, is non-increasing in $z$ on $\Omega$. 
Therefore, although $\Omega$, $u$, $v$, $A$ and $B$ satisfy the conditions of \cite[Theorem 2.7.3]{PPJS1}, the conclusion that the zero of $v-u$ on $\partial B_1(0)$ is of finite order is violated.
\end{ex}

\section{Conclusion}

In Theorem \ref{CSMP}, the outward ball condition on $\mathcal{S}$ in Definition \ref{D1} can be generalised to an outward $C^{1,Dini}$ condition, provided that the conditions on the coefficients of $L$ are appropriately constrained.
This can be achieved with more restrictive conditions in the statement of Theorem \ref{CSMP}, by replacing the function in Lemma \ref{lem1} with a suitable alternative (for instance, the regularized distance functions constructed in \cite[Sections 1 and 2]{GL1}).

In relation to Theorem \ref{t2}, a fully nonlinear version can be established without substantial additional technicality (see, for example \cite{EH1} or \cite{PPJS1}). 
Moreover, the condition bounding $v_{x_ix_j}$ can be relaxed provided that the right hand side of \eqref{a}-\eqref{d} can be expressed (for instance, by further constraining the growth of $\Lambda (d)$ as $d\to 0$) so that Theorem \ref{BPL1} can be applied. 

With regard to Theorem \ref{tfinal}, we note that allowable blow-up in $A$ and $B$ as $x\to x_b$ can be accommodated by using the more general integrability conditions on coefficients in Theorem \cite[Theorem 1.2]{NT1} i.e. by using Theorem \cite[Theorem 5.1]{NT1}.
Also, complementary results are contained in \cite{LR2}, \cite{LR1} and \cite{SDL1} where BPL for quasi-linear elliptic PDI are established under more regular domain and PDI constraints, but which guarantee the existence of non-zero (first) outward directional derivatives. 
It is also pertinent to note that in \cite{SDL1} the author highlights two further distinct incorrect statements of BPP from those highlighted here and in \cite{DNJM1}.

\section*{Acknowledgements}
The author would like to thank numerous colleagues in the School of Mathematics at the University of Birmingham, past and present, who provided helpful comments in relation to the preparation of this manuscript.


\begin{thebibliography}{ww}
\bibitem{ABMMZ1} R. Alvarado, D. Brigham, V. Maz'ya, M. Mitrea and E. Ziad\'e, `On the regularity of domains satisfying a uniform hour-glass condition and a sharp version of the Hopf-Oleinik boundary point principle.', \emph{J. Math Sci}, \textbf{176}, 3, (2011), 281-360, \href{https://doi.org/10.1007/s10958-011-0398-3}{10.1007/s10958-011-0398-3}
\bibitem{AN1} D. E. Apushkinskaya and A. I. Nazarov, `A counterexample to the Hopf-Oleinik lemma (elliptic case)' \emph{Anal. \& PDE}, \textbf{9}, 2, (2016), 439-458, \href{http://dx.doi.org/10.2140/apde.2016.9.439}{10.2140/apde.2016.9.439}
\bibitem{EH1} E. Hopf, `Element\"{a}re Bemerkungen \"{u}ber die L\"{o}sungen partieller Differentialgleichungen zweiter Ordnung vom elliptischen Typus', Siszungsberichte Preussische Akademie der Wissenschaften, Berlin, (1927), 147-152.
\bibitem{GL1} G. M. Lieberman, `Regularized distance and its applications.' Pacific J. Math., \textbf{117}, 2, (1985), 329-352, \href{https://doi.org/10.2140/pjm.1985.117.329}{10.2140/pjm.1985.117.329} 
\bibitem{AM1} A. McNabb, `Strong comparison theorems for elliptic equations of second order.' J. Math. Mech, \textbf{10}, (1961), 431-440.
\bibitem{JM1} J. Moser, `A Harnack inequality for elliptic differential equations.' \emph{Comm. Pure Appl. Math.}, \textbf{14}, (1961), 577-591, \href{https://doi.org/10.1002/cpa.3160140329}{10.1002/cpa.3160140329}
\bibitem{AIN1} A. I. Nazarov, `A centennial on the Zaremba-Hopf-Oleinik lemma', \emph{SIAM J. MATH. ANAL.}, \textbf{14}, 1, (2012), 437-453, \href{https://doi.org/10.1137/110821664}{10.1137/110821664}
\bibitem{DNJM1} D. J. Needham and J. C. Meyer, `A note on the classical weak and strong maximum principles for linear parabolic partial differential inequalities.' ZAMP, \textbf{66}, 4, (2015), 2081-2086, \href{https://doi.org/10.1007/s00033-014-0492-8}{10.1007/s00033-014-0492-8}
\bibitem{PW1} M. H. Protter and H. F. Weinberger, \emph{Maximum Principles in Differential Equations.} (Springer-Verlag, 1984, New York).
\bibitem{PPJS2} P. Pucci and J. Serrin, `The strong maximum principle revisited.' \emph{J. Differential Equations}, \textbf{196}, 1, (2004), 1-66, \href{https://doi.org/10.1016/j.jde.2003.05.001}{10.1016/j.jde.2003.05.001} 
\bibitem{PPJS1} P. Pucci and J. Serrin, \emph{The Maximum Principle} (Birkh\"{a}user, 2007, Basel).
\bibitem{PPVR1} P. Pucci and V. D. R\u{a}dulescu, `A maximum principle with a lack of monotonicity.' \emph{Electronic Journal of Qualitative Theory of differential Equations.}, 58, (2018), 1–11, \href{https://doi.org/10.14232/ejqtde.2018.1.58}{10.14232/ejqtde.2018.1.58}
\bibitem{LR2} L. Rosales, `Generalizing Hopf's boundary point lemma.' \emph{Canadian Math. Bull.}, \textbf{62}, 1, (2019), 183-197, \href{https://doi.org/10.4153/CMB-2017-074-6}{10.4153/CMB-2017-074-6}
\bibitem{LR1} L. Rosales, `A Hopf-type boundary point lemma for pairs of solutions to quasilinear equations.' \emph{Canadian Math. Bull.}, \textbf{62}, 3, (2019), 607-621, \href{https://doi.org/10.4153/S0008439519000055}{10.4153/S0008439519000055}
\bibitem{SDL1} J. C. Sabina De Lis, `Hopf maximum principle revisited.' \emph{Electronic J. Differential Equations}, \textbf{2015}, 115, (2015), 1-9.
\bibitem{JS1} J. Serrin, `On the strong maximum principle for quasilinear second order differential inequalities.' \emph{J. Functional Analysis}, \textbf{5}, (1970), 184-193, \href{https://doi.org/10.1016/0022-1236(70)90024-8}{10.1016/0022-1236(70)90024-8}
\bibitem{NT1} N. S. Trudinger, `On Harnack type inequalities and their application to quasilinear elliptic equations.' \emph{Comm. Pure. Appl Math.}, \textbf{20}, (1967), 721-747, \href{https://doi.org/10.1002/cpa.3160200406}{10.1002/cpa.3160200406}
\bibitem{Z2} L. Zaj\'{i}\v{c}ek, `Porosity and $\sigma-$porosity.' \emph{Real Analysis Exchange}, \textbf{13}, 2, (1987), 314-350.
\bibitem{Z1} L. Zaj\'{i}\v{c}ek, `On $\sigma$-porous sets in abstract spaces.' \emph{Abstr. Appl. Anal.}, \textbf{2005}, 5, (2005), 509-534, \href{https://doi.org/10.1155/AAA.2005.509}{10.1155/AAA.2005.509}
\end{thebibliography}
\end{document}